\documentclass[a4paper,10pt]{amsart}

\usepackage[utf8]{inputenc}
\usepackage{amsmath}
\usepackage{amssymb}
\usepackage{amsthm}
\usepackage{color}
\usepackage{enumitem}
\usepackage{hyperref}
\usepackage{pb-diagram}

\DeclareMathOperator{\st}{st}
\DeclareMathOperator{\Aut}{Aut}

\newcommand{\Z}{\mathbb{Z}}
\newcommand{\N}{\mathbb{N}}

\newtheorem{theorem}{Theorem}[section]
\newtheorem{lemma}[theorem]{Lemma}
\newtheorem{proposition}[theorem]{Proposition}
\newtheorem{corollary}[theorem]{Corollary}

\theoremstyle{definition}
\newtheorem{definition}[theorem]{Definition}

\theoremstyle{remark}
\newtheorem{remark}[theorem]{Remark}

\title[On the concept of fractality for some groups]{On the concept of fractality for groups  \\[3mm] of automorphisms of a regular rooted tree}
\author[J.\ Uria-Albizuri]
{Jone Uria-Albizuri}
\date{}
\thanks{The author is supported by the Basque Government research project IT753-13 and by the Basque Goverment predoctoral grant PRE-2014-1-347.}

\begin{document}

\maketitle

\begin{abstract}
   The aim of this article is to discuss and clarify the notion of fractality for subgroups of the group of automorphisms of a regular rooted tree. For this purpose we define three types of fractality. We show that they are not equivalent, by giving explicit examples. Furthermore we present some tools that are helpful in order to determine the fractality of a given group.
\end{abstract}

\section{Introduction}

The subgroups of the group of automorphisms of the $d$-adic tree $T$ (i.e.\ a regular rooted tree with $d$ branches going down at every vertex) are an important source of groups with interesting properties. For example, finitely generated torsion infinite groups can be constructed easily, giving a negative answer to the General Burnside Problem. The large amount of articles about this topic in the last years shows their interest.

Given a subgroup $G$ of $\Aut T$ the section of an element $g\in G$ at a vertex $u$ is an automorphism which represents how $g$ acts on the subtree of $T$ hanging from the vertex $u$ (the formal definition is given in Section 2).  We say that $G$ is \emph{self-similar\/} if for each element $g\in G$ and each vertex $u\in T$ the section of $g$ at the vertex $u$ belongs to $G$ again. This is a natural property that a majority of the most interesting subgroups of $\Aut T$ possess. 

It is usual to work with vertex and level stabilizers of $G$, i.e.\ the subgroups of all automorphisms in $G$ that fix a vertex $u$ or a whole level $L_n$ of the tree, denoted by $\st_G(u)$ and $\st_G(L_n)$, respectively. Then one can consider the homomorphisms $\psi_u$, which sends each $g\in\st_G(u)$ to its section at the vertex $u$, and $\psi_n$, which sends each $g\in\st_G(L_n)$ to the $d^n$-tuple of its sections at the $n$-th level. Observe that in these cases the sections are just the restrictions to the corresponding subtrees.

If $G$ is the whole group $\Aut T$, then the homomorphisms $\psi_u$ and $\psi_n$ are surjective onto $\Aut T$ and
$\Aut T \times\overset{d^n}{\dots}\times \Aut T$, respectively. 
On the other hand, if $G$ is self-similar then the images of $\psi_u$ and $\psi_n$ are contained in $G$ and $G\times \overset{d^n}{\dots}\times G$, and we will consider these sets to be the codomains of those maps.
It is natural to ask whether $\psi_u$ and $\psi_n$ are also onto in this case. For many interesting groups, $\psi_u$ is known to be onto, i.e. $\psi_u(\st_G(u))=G$ for each $u\in T$, and the group $G$ is then called \emph{fractal\/}, recurrent or self-replicating (see \cite{Bru, Gri2}). However, in general it is too strong to ask $\psi_n$ to be surjective, and we content ourselves with the image of $\psi_n$ being a subdirect product of $G\times \overset{d^n}{\dots}\times G$, namely that $\psi_u(\st_G(L_n))=G$ for each $u\in L_n$.
In some papers, this condition is only required for $n=1$; however, as we shall see, it is not always inherited by the rest of the levels.
Thus it is necessary to make a distinction between these two concepts.
Following terminology from previous papers, $G$ is said to be \emph{strongly fractal\/} or strongly self-replicating if
$\psi_x(\st_G(L_1))=G$ for all $x\in X$.
Then we say that $G$ is \emph{super strongly fractal} if $\psi_u(\st_G(L_n))=G$ for each $n\in \N$ and $u\in L_n$.

Obviously, every super strongly fractal group is also strongly fractal, and every strongly fractal group is fractal, but there is some confusion in the literature about the converse. In several papers, fractal groups are claimed to be the same as strongly fractal groups, or else fractal groups are simply introduced by using the definition of strongly fractal groups (see \cite{Bar1,Bru,Don2,Dah,Don}).  In some other papers, a distinction is made between these two concepts (see \cite{Bar2,Gri2}), but no examples can be found in the literature where a certain fractal group is shown not to be strongly fractal.
On the other hand, strongly fractal and super strongly fractal groups have not been clearly distinguished either.
Since a self-similar group that acts transitively on each level can be checked to be fractal by looking only at the vertices on the first level, one may think that the same holds for the property of being strongly fractal, see for example the paragraph after Definition 3.6 in \cite{Gri2}. This would mean that being strongly fractal and super strongly fractal are equivalent.
However, as we shall see, this is not the case.

Our aim in this article is to fill this gap.
On the one hand, for every $d\ge 3$, we give explicit examples of groups that are fractal but not strongly fractal.
More specifically, we show that a certain subgroup of the Hanoi Towers group is of this type. We remark that the restriction to $d\ge 3$ is necessary for these examples to exist, since one can easily show that for $d=2$ a fractal group is always strongly fractal.
In proving that those groups are not strongly fractal, we have obtained a couple of results that allow us to estimate the image of a level stabilizer under $\psi_u$, which may have some interest of their own.
On the other hand, we also give examples of groups which are strongly fractal but not super strongly fractal, and examples of super strongly fractal groups.
These examples belong to the class of the so-called Grigorchuk-Gupta-Sidki groups (GGS-groups, for short), which are a natural generalisation of the Grigorchuk group \cite{G}, and the Gupta-Sidki examples from \cite{GS}.

\section{Preliminaries}

Let us consider a set $X$ with $d$ elements. The $d$-adic tree $T$ is a tree whose set of vertices is the free monoid $X^*$, where a word $u$ is a descendant of $v$ if $u=vx$ for some $x\in X$. The only word of length zero is the empty word $\emptyset$, which is the root of the tree $T$. If we consider the words of length at most $n$ we have a finite subtree $T_n$, and the words whose length is exactly $n$ form the $n$-th level of the tree, $L_n$. 

An automorphism of the $d$-adic tree is a map preserving incidence between vertices. All automorphisms of $T$ form a group $\Aut T$ under composition, where we write $fg$ for $g\circ f$. Thus $(fg)(u)=g(f(u))$ for every vertex $u$ of $T$.

Let us consider the natural projection $\pi_n:\Aut T\longrightarrow \Aut T_n$, which sends every automorphism to its restriction to $T_n$. Observe that the stabilizer $\st(L_n)$ of the $n$-th level is the kernel of $\pi_n$, so it is a normal subgroup in $\Aut T$, and we have $\Aut T_n \cong \Aut T/ \st(L_n)$.

An important observation is that every automorphism $g\in\Aut T$ can be fully described by saying for each vertex $u\in T$ how  $g$ permutes the $d$ vertices hanging from $u$. So, there is a permutation $\alpha$ of $X$ such that $g(ux)=g(u)\alpha(x)$. We say that $\alpha$ is the \textbf{label} of $g$ at the vertex $u$, and we denote it by $g_{(u)}$.

Since $T\cong T_u$, where $T_u$ denotes the subtree hanging from a vertex $u$, we have $\Aut T\cong \Aut T_u$. We speak about the \textbf{section} of $g$ at the vertex $u$ and we denote it by $g_u$, to refer to the automorphism defined by $g(uv)=g(u)g_u(v)$ for each vertex $v$. Then we have the following formulas:

\begin{align}\label{formulas}
\nonumber
(f^{-1})_{u}&=(f_{f^{-1}(u)})^{-1},\\ 
(fg)_{u}&=f_{u}g_{f(u)},\\
\nonumber
f_{uv}&=(f_u)_v,
\end{align}
and,
\begin{equation}
\label{conjsect}
(f^g)_{u}=(g_{g^{-1}(u)})^{-1}f_{g^{-1}(u)}g_{g^{-1}f(u)}.
\end{equation}

Also, we define the homomorphism $$\psi_n:\st(L_n)\longrightarrow\Aut T\times\overset{d^n}{\cdots}\times\Aut T$$ which sends $g\in\st(L_n)$ to the $d^n$-tuple of its sections $(g_{u_1},\cdots, g_{u_{d^n}})$, with $u_i\in L_n$. In the same way, for the stabilizer $\st(u)$ of the vertex $u$, we have a homomorphism denoted by $\psi_u$ which sends $g\in\st(u)$ to $g_u$.

Sometimes it is useful to think of $\Aut T$ as a semidirect product. 
\begin{proposition}\label{prop1}
Let $T$ be the $d$-adic tree and let us consider the following subgroup for each $n\in \N$: $$H_n=\{h\in \Aut T\mid h_{u}=1\,\, \forall u\in L_n\}.$$ Then we have  $$\Aut T=H_n\ltimes \st(L_n).$$  
\end{proposition}

Observe that for $f\in \st(L_n)$ and $g=hg'\in \Aut T$, with $h\in H_n$ and $g'\in\st(L_n)$, we have
\begin{equation}\label{conj2}
(f^g)_u=(f_{h^{-1}(u)})^{g_u}=(f_{h^{-1}(u)})^{g'_u}\text{ for all $u\in L_n$.}
\end{equation}

Let now $G\leq \Aut T$. Then we can consider the stabilizers in $G$ of each vertex, $\st_G(u)=\st(u)\cap G$, and the level stabilizers $\st_G(L_n)=\cap_{u\in L_n}\st_G(u)=\st(L_n)\cap G$. So we have the restrictions of $\psi_n$ and $\psi_u$ to $\st_G(L_n)$ and $\st_G(u)$, respectively. Since we are interested in those groups for which the images under $\psi_u$ and $\psi_n$ are in $G$ and $G\times\overset{d^n}{\dots}\times G$, we give the following definition.
\begin{definition}
We say that a group $G\leq \Aut T$ is \emph{self-similar} if for each element of $G$ its sections are also elements of $G$, in other words, if \begin{equation}\label{self2}
\{g_u\mid g\in G,\,\,\, u\in T\}\subseteq G.
\end{equation}
\end{definition}
It is easy to prove by induction on the length of a vertex and using the first two formulas in (\ref{formulas}), that if (\ref{self2}) is satisfied by the vertices of the first level the group is self-similar (see Proposition 3.1 in \cite{Gri}).
\begin{lemma}\label{self}
A group $G=\langle S\rangle\leq \Aut T$ is self-similar if and only if $s_x\in G$ for each $s\in S$ and $x\in X$.
\end{lemma}

Observe that even if in the case of the whole group of automorphisms $\Aut T$ the homomorphisms $\psi_n$ and $\psi_u$ are surjective, this might not be true in general. According to this we have the following definitions.

\begin{definition}
Let $G\leq \Aut T$ be a self-similar group.
\begin{enumerate}
\item[(i)] We say that $G$ is \emph{fractal} if $\psi_u(\st_G(u))=G$ for each vertex $u\in T$.
\item[(ii)] We say that $G$ is \emph{strongly fractal} if $\psi_x(\st_G(L_1))=G$ for each $x\in X$.
\item[(iii)] We say that $G$ is \emph{super strongly fractal} if $\psi_u(\st_G(L_n))=G$ for each $u\in L_n$ and each $n\in \N$.
\end{enumerate}
\end{definition}
Notice that the definition of being super strongly fractal does not imply that $\psi_n$ is surjective from $G$ to $G\times\overset{d^n}{\dots}\times G$, but only that $\psi_n(\st_G(L_n))$ is a subdirect product in $G\times\overset{d^n}{\dots}\times G$. The same remark applies to strongly fractal groups with $n=1$.

There is a special case in which the first two definitions are equivalent.

\begin{lemma}\label{root}
Let $G\leq \Aut T$ and consider   a $d$-cycle $\sigma\in S_X$. If for each $g\in G$ we have $g_{(\emptyset)}=\sigma^k$ for some $k\in\N$  and $G$ is fractal, then $G$ is strongly fractal.
\end{lemma}
\begin{proof}
Let $g\in\st_G(x)$ for $x\in X$. Then $\sigma^k(x)=x$ which only happens if $k\equiv 0\pmod d$. This implies that $g\in \st_G(L_1)$, so $\st_G(x)=\st_G(L_1)$, and we have finished.
\end{proof}
Observe that for $d=2$ the label at the root must be $1$ or $(1\,2)$, so according to the previous lemma, in this case being fractal and being strongly fractal are equivalent.

This can be generalised, to obtain another important corollary that follows from the previous lemma in the case $d=p$ where $p$ is a prime. If we consider $T$ to be the $p$-adic tree, $\Aut T$ is a profinite group which has a standard Sylow pro-$p$ subgroup consisting of automorphisms which have powers of a fixed $p$-cycle as a label in every vertex. Then, the previous lemma shows that for every subgroup of the Sylow pro-$p$ subgroup  being fractal and strongly fractal are equivalent. For example, this happens for the GGS-groups (for the definition see Section 4). 

One of our goals is to give examples of subgroups of $\Aut T$ for $d\ge 3$ which are fractal but are not strongly fractal. We next give the definition of being level transitive, because the examples that we present are of this type and also because in this case it is easier to check if a group is fractal or not. 
\begin{definition}
Let $G\leq \Aut T$. We say that $G$ is \emph{level transitive}  or that acts spherically transitively on $T$, if it is transitive on each level.
\end{definition}

In a similar way to Lemma \ref{self}, in some cases to check whether a group is fractal it is enough to look at the vertices on the first level (for a reference, see Section 3 in \cite{Gri2}).

\begin{lemma}\label{transfractal2}
If $G\leq \Aut T$ is transitive on the first level and $\psi_x(\st_G(x))=G$ for some $x\in X$, then $G$ is fractal and level transitive.
\end{lemma}
Since we will want to prove that a group is not strongly fractal, we are interested in identifying which the first level stabilizer is. We present a tool that we have developed in order to do this in the following lemma. Let us denote by $\rho$ the homomorphism from $G$ to $S_d$ that sends each $g\in G$ to the label of $g$ at the root, $g_{(\emptyset)}$.
\begin{lemma}\label{presentation}
Let $G\leq\Aut T$ and put $J=\rho(G)$. Suppose that we have a presentation $J=\langle Y\mid R\rangle$ and let $\theta:F\longrightarrow J$ be the epimorphism corresponding to this presentation, where $F$ is the free group generated by $Y$. If there exists  a surjective homomorphism $\phi:F\longrightarrow G$ making the following diagram commutative,
\[
\begin{diagram}
\node{F}\arrow{e,t}{\phi}\arrow{se,b}{\theta}\node{G}\arrow{s,r}{\rho}\\
\node[2]{J}
\end{diagram}
\]
then, $$\st_G(L_1)=\langle\phi(R)\rangle^G.$$
\end{lemma}
\begin{proof}
We know that $\ker\theta=\langle R\rangle ^F$. On the other hand, since $\phi$ is surjective, every $g\in G$ can be written as $g=\phi(x)$ for some $x\in F$, and then $g\in\ker\phi$ if and only if $x\in\ker(\rho\circ\phi)$.
Consequently,
\begin{align*}
 \st_G(L_1)&=\ker\rho=\phi(\ker(\rho\circ\phi))\\
 &=\phi(\ker\theta)=\phi(\langle R\rangle^F)\\
 &=\langle\phi(R)\rangle^G.
\end{align*}

\end{proof}
Notice that the actual condition we are asking about $\phi$ is to be surjective, because by the universal property of free groups we are always able to construct some $\phi$ making the diagram commutative. In other words, the point is whether for each $y\in Y$ we can choose an element $g_y\in\rho^{-1}(\theta(y))$, in such a way that $\{g_y\mid y\in Y\}$ generates the whole group $G$ or not.

Now, in the following lemma we present another new result, which will help us to prove that the image of a level stabilizer under $\psi_u$ is strictly contained in $G$.
\begin{lemma}\label{lema2}
Let $G\leq\Aut T$ be a self-similar group. If $K=\langle S\rangle ^G\subseteq \st_G(L_n)$ for some $n\in \N$ and $\psi_u(S)\subseteq N$ for each $u\in L_n$, where $N\trianglelefteq G$, then $\psi_u(K)\subseteq N$ for each $u\in L_n$.
\end{lemma}
\begin{proof}
Consider $k\in K$ and let us write $k=(s_1^{\epsilon_1})^{g_1}\dots(s_r^{\epsilon_r})^{g_r}$ where $\epsilon_i\in\{-1,1\}$, $s_i\in S$ and $g_i\in G$ for each $i=1,\dots,r$. Let $u\in L_n$. Since $K\leq \st_G(L_n)$ we know that $k\in \st_G(u)$ and we have $$\psi_u(k)=\psi_u(s_1^{g_1})^{\epsilon_1}\dots \psi_u(s_r^{g_r})^{\epsilon_r}.$$
Thus it is enough to see that $\psi_u(s^g)\in N$ for each $s\in S$, $g\in G$.
 Since $G\leq \Aut T$ and $\Aut T=H_n\ltimes\st(L_n)$ we write each $g=ht$ where $h\in H_n$ and $t\in \st(L_n)$. Now by (\ref{conj2}) we have $$\psi_u(s^g)=(s_{h^{-1}(u)})^{t_u}$$ 
for each $u\in L_n$, and since $\psi_{h^{-1}(u)}(S)\subseteq N$ and $N$ is normal in $G$, it is enough to check that $t_u$ belongs to $G$. We know that $G$ is self-similar, so $g_{v}\in G$ for each $v\in T$, in particular for $v=h^{-1}(u)$, but $g_{h^{-1}(u)}=h_{h^{-1}(u)}t_u=t_u$ because $h\in H_n$, so we are done.
\end{proof}


Now, let us introduce a stronger version of the previous lemma that will help us to check whether a strongly fractal group is super strongly fractal or only strongly fractal.

\begin{lemma}
Let $G$ be level transitive and super strongly fractal. If $K=\langle S\rangle^G\subseteq \st_G(L_n)$ for some $n\in \N$, then $\psi_u(K)=\langle\psi_v(S)\mid v\in L_n\rangle^G$ for any $u\in L_n$.
\end{lemma}
\begin{proof}
Let us denote $N=\langle\psi_v(S)\mid v\in L_n\rangle^G$.
Since $\psi_u(S)\subseteq N$ for every $u\in L_n$, which is a normal subgroup, the inclusion $\psi_u(K)\subseteq N$ follows from the previous lemma.

Now, let $g=(\psi_{u_1}(s_1)^{\epsilon_1})^{g_1}\dots(\psi_{u_r}(s_r)^{\epsilon_r})^{g_r}\in N$.  Since $G$ is level transitive for every $u_i\in\{u_1,\dots,u_r\}$, there is some $f_i\in G$ such that $f_i(u_i)=u$. Then, by (\ref{conj2}) $$\psi_u(s_i^{f_i})^{((f_i)_{u_i})^{-1}}=\psi_{u_i}(s_i).$$ Then we can write $g=(\psi_{u}(s_1^{f_1})^{\epsilon_1})^{g'_1}\dots(\psi_{u}(s_r^{f_r})^{\epsilon_r})^{g'_r}$, where $g'_i=((f_i)_{u_i})^{-1}g_i\in G$. From the fact that $G$ is super strongly fractal, we know that there are some $h_i\in\st_G(L_n)$ such that $\psi_u(h_i)=g'_i$ for $i=1,\dots,r$. We conclude because
\begin{align*}
g&=(\psi_u(s_1^{f_1})^{\epsilon_1})^{g'_1}\dots(\psi_u(s_r^{f_r})^{\epsilon_r})^{g'_r}\\
&=(\psi_u(s_1^{f_1})^{\epsilon_1})^{\psi_u(h_1)}\dots(\psi_u(s_r^{f_r})^{\epsilon_r})^{\psi_u(h_r)}\\
&=\psi_u((s_1^{\epsilon_1})^{f_1h_1}\dots(s_r^{\epsilon_r})^{f_rh_r})\in\psi_u(K).
\end{align*}
\end{proof}
In particular, we have the following result when the group is strongly fractal.
\begin{corollary}\label{strongly}
Let $G$ be a strongly fractal group which acts transitively on the first level. If $K=\langle S\rangle^G$ and $K\subseteq \st_G(L_1)$ then $\psi_x(K)=\langle\psi_y(S)\mid y\in X\rangle^G$ for any $x\in X$.
\end{corollary}

Finally let us introduce another lemma that will help us to prove that a group is super strongly fractal. This lemma tells us that in some cases, it suffices to check whether in each level stabilizer there are elements whose sections at vertices on this level generate the whole group. 

\begin{lemma}\label{stabrooted}
Let $G\leq\Aut T$ be a self-similar group such that there is a rooted automorphism $a\in G$, with $a_{(\emptyset)}$ a $d$-cycle. If for each $n\in \N$ we have $\langle\psi_{u_n}(\st_G(L_n))\mid u_n\in L_n\rangle=G$, then $G$ is super strongly fractal.
\end{lemma}
\begin{proof}
The proof works by induction on the length of the vertices. Let $x\in X$ and $g\in G$. We know that there are some $y_1,\dots,y_r\in  X$ such that $g=\psi_{y_1}(g_1)^{\epsilon_1}\dots\psi_{y_r}(g_r)^{\epsilon_r}$, where $g_i\in\st_G(L_1)$ and $\epsilon_i\in\{1,-1\}$. Then for each $i=1,\dots,r$ we have $a^{j_i}(y_i)=x$ for some $j_i\in\{0,\dots,d-1\}$. Then considering $g_i^{a^{j_i}}$ we get an element on the first level stabilizer such that $(g_i^{a^{j_i}})_x=(g_i)_{y_i}$. Then the element $h=(g_1^{a^{j_1}})^{\epsilon_1}\dots (g_r^{a^{j_r}})^{\epsilon_r}\in\st_G(L_1)$ satisfies $h_x=g$, so $\psi_x(\st_G(L_1))=G$.

Now let us suppose that we know the result for length $n-1$ and let us see it for $n$.
Let $v=x_1\dots x_{n}$ and $g\in G$. By assumption we know that $g=\psi_{w_1}(g_1)^{\epsilon_1}\dots\psi_{w_r}(g_r)^{\epsilon_r}$ where $w_i\in L_n$, $g_i\in\st_G(L_{n})$ and $\epsilon_i\in\{1,-1\}$ for each $i=1,\dots,r$. It suffices to show that for $i=1,\dots, r$ there is some $h_i\in\st_G(L_{n})$ such that $(h_i)_v=(g_i)_{w_i}$, because then $h=h_1^{\epsilon_1}\dots h_r^{\epsilon_r}\in\st_G(L_{n})$ and $h_v=g$, as desired.

Let $w$ be an arbitrary vertex in $L_n$. Then $w=y_1\dots y_{n}$ with $y_i\in X$. For each $k=1,\dots,n$ there is some $j_k=0,\dots,d-1$ such that $a^{j_k}(y_k)=x_k$. By inductive assumption $a\in\psi_{u}(\st_G(L_k))$ for every $u\in L_k$, with $k=1,\dots,n-1$. Thus, for each $k=1,\dots,n-1$ there is some $f_k\in\st_G(L_{k})$ such that $(f_k)_{y_1\dots y_k}=a^{j_{k+1}}$. Then if we consider the element $f=a^{j_1}f_1\dots f_{n-1}$, which belongs to $H_n$, we obtain that $f(w)=v$. Thus, in particular for each $i=1,\dots,r$ there is some $t_i\in H_n$ such that $t_i(w_i)=v$. Then $h_i=g_i^{t_i}\in\st_G(L_n)$ and by (\ref{conj2}) $$(h_i)_{v}=(g_i)_{t_i^{-1}(v)}=(g_i)_{w_i}.$$
\end{proof}

\begin{remark}\label{onevertex}
In particular, in the conditions of the previous lemma, it is enough for a group $G$ to be super strongly fractal having  one vertex $u_n\in L_n$ such that $\psi_{u_n}(\st_G(L_n))=G$ for each $n\in \N$.
\end{remark}
\section{Fractal groups which are not strongly fractal}

In this section we present an example for each $d\geq 3$ which is fractal but not strongly fractal. Even more, that example is a group acting spherically transitively on $T$. We denote by $x_i$ for $i=1,\dots, d$ the elements of $X$, or what it is the same, the vertices of the first level.

The example that we consider is a subgroup of the Hanoi Towers group, which is defined as follows for each $d\geq 3$.

For $1\leq i<j\leq d$, we define the element $a_{ij}$ which has the permutation $(x_i\, x_j)$ at the root and for each vertex on the first level:
\begin{equation*}
(a_{ij})_{x_k}=\begin{cases}
1 & \text{ if $k=i,j$}\\
a_{ij} & \text{ else.}
\end{cases}
\end{equation*}
The Hanoi Towers group is $H=\langle a_{ij}\mid 1\leq i<j\leq d\rangle$. Although $H$ is strongly fractal (see \cite[page 13]{Gri}), we are going to show that it has a subgroup which is fractal but not strongly fractal.

We consider the subgroup $G=\langle a_{i,i+1}\mid i=1,\dots,d-1\rangle\leq H$. To simplify the notation, we write $b_i=a_{i,i+1}$.

As a consequence of Lemma \ref{self} it is clear that $G$ is self-similar, because $(b_j)_{x_i}\in G$ for each $j=1,\dots,d-1$ and $i=1,\dots, d$.
 
Let us see that $G$ is fractal. Observe that since the element $b_{d-1}b_{d-2}\dots b_1$ has the label $(x_1x_2...x_d)$ at the root, $G$ is transitive on the first level, so by Lemma \ref{transfractal2} it is enough to show that $\psi_{x_1}(\st_G(x_1))=G$.

It suffices to check that each $b_{i}\in\psi_{x_1}(\st_G(x_1))$. Since $b_i\in\st_G(x_1)$ for $i\neq 1$ and in this case $\psi_{x_1}(b_i)=b_i$, it only remains to check that $b_1\in\psi_{x_1}(\st_G(x_1))$. 
To show this, consider the element $b_1^{b_2b_1}$. First of all observe that $(b_1^{b_2b_1})_{(\emptyset)}=(x_1 x_2)^{(x_1 x_2 x_3)}=(x_2 x_3)$, so $b_1^{b_2b_1}$ belongs to $\st_G(x_1)$. On the other hand, using (\ref{conjsect}) we have
\begin{eqnarray*}
(b_1^{b_2b_1})_{x_1}&=&((b_2b_1)_{(b_2b_1)^{-1}(x_1)})^{-1}(b_1)_{(b_2b_1)^{-1}(x_1)}(b_2b_1)_{(b_2b_1)^{-1}b_1(x_1)}\\
&=&((b_2b_1)_{x_3})^{-1}(b_1)_{x_3}(b_2b_1)_{x_3}\\
&=&((b_2)_{x_3}(b_1)_{x_2})^{-1}b_1(b_2)_{x_3}(b_1)_{x_2}\\
&=&b_1.
\end{eqnarray*}
We obtain that $\psi_{x_1}(b_{1}^{b_{2}b_{1}})=b_{1}$. Thus, we conclude that $\psi_{x_1}(\st_G(x_1))=G$ as desired.

Let us now calculate $\st_G(L_1)$. We have $\rho(G)=\langle\rho(b_i)\mid i=1,\dots,d-1\rangle=S_d$.  We know that a presentation of the group $S_d$ can be obtained by considering as generators $\{\tau_i=(i\,i+1)\}_{i=1,\dots,d-1}$ and the following relations:
\begin{eqnarray*}
&\tau_i^2=1,& i=1,\dots, d-1,\\
&\tau_i\tau_j=\tau_j\tau_i, &|i-j|>1,\\
&(\tau_i\tau_{i+1})^3=1,&i=1,\dots, d-2.
\end{eqnarray*} 
The proof of this fact can be found in \cite[page 296]{Suu}.

In order to apply Lemma \ref{presentation}, let $F$ be the free group generated by $\{\tau_1,\dots,\tau_{d-1}\}$ and $\theta:F\longrightarrow S_d$ the epimorphism corresponding to the presentation above. Thus $\ker\theta=\langle \tau_i^2,[\tau_i,\tau_j],(\tau_i\tau_{i+1})^3\mid i,j=1,\dots,d-1, |i-j|>1\rangle^F$. For each $i=1,\dots, d-1$ we have $b_i\in\rho^{-1}(\theta(\tau_i))$ and the $b_i$ generate the whole group $G$. We can define $\phi:F\longrightarrow G$ by sending $\tau_i$ to $b_i$ for each $i=1,\dots,d-1$. Then $\phi$ is a surjective homomorphism that makes the diagram commutative. Now, applying the lemma,
 if $$S=\{\{b_i^2\}_{i=1,\dots,d-1},\{(b_ib_{i+1})^3\}_{i=1,\dots,d-2},\{\left[b_i,b_j\right]\}_{|i-j|>1}\},$$ 
then we obtain that
$$\st_G(L_1)=\langle S\rangle^G.$$

Let us see to conclude that $\psi_{x_k}(\st_G(L_1))\neq G$ for some $k=1,\dots, d$. In fact we will see that this happens for any $k\in\{1,\dots,d\}.$
 
One can check that
\begin{equation*}
(b_{i}^2)_{x_k}=\begin{cases}
b_i^2 & \text{ if $k\neq i,i+1$},\\
1 & \text{ if $k=i, i+1$},\\
\end{cases}
\end{equation*}

\begin{equation*}
((b_{i}b_{i+1})^3)_{x_k}=\begin{cases}
(b_{i}b_{i+1})^3 & \text{ if $k\neq i,i+1,i+2$},\\
b_{i}b_{i+1}& \text{ if $k=i$},\\
b_{i+1}b_{i}& \text{ if $k=i+1$},\\
b_{i}b_{i+1}& \text{ if $k=i+2$},\\
\end{cases}
\end{equation*}

and, for $|i-j|>1$,
\begin{equation*}
(\left[b_i,b_j\right])_{x_k}=
\begin{cases}
\left[b_i,b_j\right] & \text{ if $k\neq i,i+1,j,j+1$},\\
1& \text{ else}.
\end{cases}
\end{equation*}
To see the importance of the condition $|i-j|>1$ in the last case, let us calculate for example $\left[b_i,b_j\right]_{x_i}$.
\begin{eqnarray*}
\left[b_i,b_j\right]_{x_i}&=&(b_i^{-1}b_i^{b_j})_{x_i}\\
&=&(b_i^{-1})_{x_i}(b_i^{b_j})_{x_{i+1}}\\
&=&((b_j)_{b_j^{-1}(x_{i+1})})^{-1}(b_i)_{b_j^{-1}(x_{i+1})}(b_j)_{b_j^{-1}b_i(x_{i+1})}\\
&=&((b_j)_{x_{i+1}})^{-1}(b_i)_{x_{i+1}}(b_j)_{x_i}\\
&=&b_j^{-1}b_j=1.
\end{eqnarray*}
Observe that here it is important that $b_j$ does not move $x_i$ and $x_{i+1}$, which happens because $|i-j|>1$.

If $\sigma:S_d\longrightarrow\{1,-1\}$ is the homomorphism sending each permutation to its signature, observe that for any $s\in S$ and $k=1,\dots,d$ we have $\sigma(\psi_{x_k}(s)_{(\emptyset)})=1$ because $\psi_{x_k}(s)$ is always a product of an even number of $b_i$. Then, if we consider $N=\langle\psi_{x_k}(S)\mid k=1,\dots,d\rangle ^G$ we still have that $\sigma(n_{(\emptyset)})=1$ for any $n\in N$. 

Now, we have $\st_G(L_1)=\langle S\rangle ^G$ and $\psi_{x_k}(S)\subseteq N$ where $N$ is normal in $G$, so by Lemma \ref{lema2} we conclude that $\psi_{x_k}(\st_G(L_1))\subseteq N$. But $N$ can not be the whole group $G$ because each $n\in N$ has an even permutation at the root and consequently $b_i\notin N$ for each $i=1,\dots, d-1$. In other words, $\rho(N)\subseteq A_d$ while $\rho(G)=S_d$, so $N\neq G$.

\section{Strongly fractal groups which are not super strongly fractal}

In order to see an example of a group which is strongly fractal but not super strongly fractal, we have to introduce the GGS-groups. These groups are subgroups of $\Aut T$ where $T$ is the $d$-adic tree for $d\geq 2$.

\begin{definition}
Let us consider the rooted automorphism corresponding to $(1\dots d)$ and denote it by $a$. Given a non-zero vector $e=(e_1,...,e_{d-1})\in (\Z/d\Z)^{d-1}$, we define an automorphism $b\in \st(L_1)$ by means of $\psi(b)=(a^{e_1},\dots,a^{e_{d-1}},b)$. Then, a \emph{GGS-group} is the group $G$ generated by these two automorphisms $a$ and $b$.
\end{definition}

From now on we consider $d=p$ where $p$ is a prime. First of all let us see that every GGS-group is strongly fractal.

For these groups $\st_G(L_1)=\langle b\rangle ^G=\langle b,b^a,\dots b^{a^{p-1}}\rangle$. To simplify notation, we write $b_i=b^{a^{i}}$.

\begin{lemma}
Let $G$ be a GGS-group. Then $G$ is strongly fractal.
\end{lemma}
\begin{proof}
Let us see that $G$ is fractal. Since $G$ is in the  Sylow pro-$p$ subgroup of $\Aut T$ corresponging to the cycle $(1\dots p)$, this is enough to show that $G$ is strongly fractal because of the discussion after Lemma \ref{root}.
Since $\langle a\rangle$ acts transitively on the first level, according to Lemma \ref{transfractal2} it suffices to show that $\psi_x(\st_G(x))=G$ for some $x$ in the first level. Observe that conjugating $b$ by powers of $a$ permutes the sections of $b$ at the first level. In other words,
$$\psi(b_i)=(a^{e_{p-i+1}},\dots,a^{e_{p-1}},b,a^{e_1},\dots,a^{e_{p-i}}).$$
Then, since $e$ is non-zero, there is some $e_{p-i+1}\neq 0$ and since $b_1,b_i\in\st_G(x_1)$ we obtain that $\psi_{x_1}(\st_G(x_1))\geq\langle b,a^{e_{p-i+1}}\rangle =G$. We conclude that $G$ is strongly fractal.
\end{proof}

Let us consider a GGS-group with constant defining vector. By replacing $b$ with a suitable power of $b$, we may assume that $e=(1,\dots,1)$.

\begin{proposition}
Let $G$ be a GGS-group with constant defining vector. Then $G$ is strongly fractal but not super strongly fractal.
\end{proposition}
\begin{proof}
By the previous lemma it is enough to show that $G$ is not super strongly fractal.
In \cite[Theorem 2.4]{Alc} it is shown that $|G:\st_G(L_2)|=p^{t+1}$ where $t$ is the rank of the circulant matrix which has as first row the vector $(1,\dots,1,0)$. In this case the rank is $p$. It is also proved in \cite[Theorem 2.14]{Alc} that $|G:\st_G(L_1)'|=p^{p+1}$. The mentioned article is written for $p$ an odd prime, but this two results are also true for $p=2$. Since $\st_G(L_1)/\st_G(L_2)$ is abelian we know that $\st_G(L_1)'\subseteq\st_G(L_2)$, so we conclude that $\st_G(L_2)=\st_G(L_1)'$. Now $\st_G(L_1)'=\langle [b_i,b_j]\mid i,j=1,\dots, p\rangle^G$. Observe that $\psi([b_i,b_j])=(1,\dots,1,\underset{j}{[a,b]},1,\dots 1,\underset{i}{[b,a]},1\dots,1)$. By Corollary \ref{strongly}
we conclude that
$$
\psi_{x_1}(\st_G(L_2))=\psi_{x_1}(\st_G(L_1)')=\langle[a,b],[b,a]\rangle^G=G'.$$
Now again, $\psi([a,b])=\psi(b_1^{-1}b)=(b^{-1}a,1,\dots,1, a^{-1}b)$. By the same argument as before, we have 
\begin{align*}
\psi_{x_1}(G')&=\psi_{x_1}(\langle\left[a,b\right]\rangle^G)\\
&=\langle b^{-1}a\rangle^G.
\end{align*}
But then, for the vertex $u=x_1x_1\in L_2$ we have that $\psi_{u}(\st_G(L_2))=\langle b^{-1}a\rangle^G$. It is not hard to see that $G/G'\cong C_p\times C_p$ (see \cite[Theorem 2.1]{Alc}). Since the image of $\langle ba^{-1}\rangle^G$ in $G/G'$ is cyclic, we have $\langle ba^{-1}\rangle^G\neq G$, and $G$ is not super strongly fractal. 
\end{proof}

\section{Groups which are super strongly fractal}
In the same family of GGS-groups, we have examples of groups which are super strongly fractal. More specifically, the GGS-groups which are periodic, which is equivalent to having defining vector $e$  such that $e_1+\dots +e_{p-1}\equiv 0\pmod p$ (see \cite[Theorem 1]{Vov}) are an example of super strongly fractal groups.

\begin{proposition}
Let $G$ be a GGS-group with defining vector $e=(e_1,\dots,e_{p-1})$ such that $e_1+\dots+e_{p-1}\equiv0\pmod p$. Then $G$ is super strongly fractal.
\end{proposition}
\begin{proof}
By \cite[Lemma 3.3]{Alc} we know that for $n\geq 3$ we have $\psi(\st_G(L_n))=\st_G(L_{n-1})\times\overset{p}{\dots}\times\st_G(L_{n-1})$. Since we also know that $\psi_x(\st_G(L_1))=G$ for every $x\in X$, it only remains to show that $\psi_x(\st_G(L_2))=\st_G(L_1)$ for each $x\in X$. Since $G$ contains the rooted automorphism $a$, by Remark \ref{onevertex}, it is enough to check the condition in one vertex. 

Let us consider the element $g=b_1b_2\dots b_{p-1}b$. We have 
\begin{align*}
\psi(g)&=(ba^{e_1+\dots+e_{p-1}},a^{e_1+\dots+e_{p-1}}b_{e_2+\dots+e_{p-1}},a^{e_1+\dots+e_{p-1}}b_{e_3+\dots+e_{p-1}},\dots,a^{e_1+\dots+e_{p-1}}b)\\
&=(b,b_{e_2+\dots+e_{p-1}},b_{e_3+\dots+e_{p-1}},\dots,b),
\end{align*} 
so we conclude that $g\in\st_G(L_2)$. On the other hand, since $G$ is strongly fractal by Corollary \ref{strongly} we have $$\psi_{x_1}(\st_G(L_2))\geq \psi_{x_1}(\langle g\rangle^G )=\langle b,b_{e_2+\dots+e_{p-1}},b_{e_3+\dots+e_{p-1}},\dots,b\rangle^G=\st_G(L_1)$$, and we have finished.
\end{proof}

In \cite[page 85]{Gri2}, it is said that being strongly fractal implies being super strongly fractal, and also that the first Grigorchuk group is an example of this. It is true that the first Grigorchuk group is super strongly fractal, but it is not a direct consequence of being strongly fractal. The proof of this is similar to the previous example.

\begin{definition}
Let $T$ be the $2$-adic tree. The first Grigorchuk group, denoted by $\mathcal{G}$, is the group generated by the following automorphisms $a,b,c,d$:
\begin{eqnarray*}
a_{(\emptyset)}=(12), &\psi(a)=(1,1),&\\
b,c,d\in\st(L_1),& \psi(b)=(a,c),&\\
&\psi(c)=(a,d),&\\
&\psi(d)=(1,b).&
\end{eqnarray*}
\end{definition}

\begin{proposition}
The group $\mathcal{G}$ is super strongly fractal.
\end{proposition}
\begin{proof}
In \cite[Theorem 4.3]{Bar1} it is shown that $\psi(\st_{\mathcal{G}}(L_n))=\st_{\mathcal{G}}(L_{n-1})\times \st_{\mathcal{G}}(L_{n-1})$ for $n\geq 4$. Since $a\in G$ by Lemma \ref{stabrooted} it suffices to show that $\langle\psi_{u_n}(\st_{\mathcal{G}}(L_n))\mid u_n\in L_n\rangle=\mathcal{G}$ when $n=1,2,3$.

For $n=1$ it follows from the definition of the elements $b,c,d$. Let us see the cases $n=2$ and $n=3$.

It is easy to calculate and check that $d,(ab)^4,(ac)^4\in\st_{\mathcal{G}}(L_2)$ and that
\begin{align*}
\psi_{x_2x_1}(d)&=a,\\
\psi_{x_2x_2}(d)&=c,\\
\psi_{x_2x_2}((ab)^4)&=ad,\\
\psi_{x_2x_2}((ac)^4)&=b.
\end{align*}

To conclude, the element $g=(ab)^4(adabac)^2$ belongs to $\st_{\mathcal{G}}(L_3)$ and
\begin{align*}
\psi_{x_1x_2x_1}(g)&=d,\\
\psi_{x_1x_2x_2}(g)&=ba,\\
\psi_{x_2x_2x_1}(g)&=a,\\
\psi_{x_2x_2x_2}(g)&=c.
\end{align*}
This proves that $\mathcal{G}$ is super strongly fractal.
\end{proof}

\bibliographystyle{alpha}
\bibliography{main}
\end{document}